\documentclass[a4paper,12pt]{amsart}
\usepackage{lmodern}
\usepackage{amsmath,amsfonts,amssymb,mathrsfs,amsthm}

\newtheorem{theorem}[subsection]{Theorem}
\newtheorem{lemma}[subsection]{Lemma}

\newtheorem{definition}[subsection]{Definition}

\newcommand{\dist}{{\mathop{\rm dist}}}

\newcommand{\eps}{\varepsilon}

\newcommand{\bN}{{\mathbb{N}}}

\newcommand{\bra}{\langle}
\newcommand{\ket}{\rangle}

\begin{document}

% Declarations for Front Matter

\title[Operator algebras converging to odd spheres]{Sequences of operator algebras converging to odd spheres in the quantum Gromov-Hausdorff distance}

\author[Bhattacharyya]{Tirthankar Bhattacharyya}
\author[Singla]{Sushil Singla}
\address{Department of Mathematics, Indian Institute of Science, CV Raman Road, Bengaluru, Karnataka 560012.}
\email{tirtha@iisc.ac.in, ss774@snu.edu.in}

\maketitle

\begin{abstract}

Marc Rieffel had introduced the notion of the quantum Gromov-Hausdorff distance on compact quantum metric spaces and found a sequence of matrix algebras that converges to the space of continuous functions on $2$-sphere in this distance. One finds applications of similar approximations in many places in the theoretical physics literature. In this paper, we have defined a compact quantum metric space structure on the sequence of Toeplitz algebras on generalized Bergman spaces and have proved that the sequence converges to the space of continuous function on odd spheres in the quantum Gromov-Hausdorff distance.
\end{abstract}

\setlength{\parindent}{0pt}
\setlength{\parskip}{1.6ex}

\pagestyle{headings}

%%%%%%%%%%%%%%%%%%%%%%%%%
\section{Introduction}

Rieffel introduced the notion of a \emph{compact quantum metric space}.

\begin{definition}[\cite{survey}, Definition 1.1]\label{1} Let $A$ be an order-unit space with identity element $e_A$, and let $L$ be a seminorm on $A$ taking finite values. $L$ is known as a \emph{Lip-norm} if
\begin{enumerate}
\item $L(e_A) = 0$,
\item The topology on the state space $\mathcal S(A)$ of $A$ from the metric $$\rho_L(\mu, \nu)=\sup\{|\mu(a)-\nu(a)| : L(a)\leq 1\},$$ coincides with the $weak^*$ topology on $\mathcal S(A)$.
\end{enumerate}
A compact quantum metric space is a pair $(A, L)$, where $A$ is an order-unit space and $L$ is a Lip-norm $A$.
\end{definition}

The justification for the name `compact quantum metric space' comes from the fact that if $(X, d)$ is a compact metric space, then the Lipschitz seminorm $L_d$ is a Lip-norm on the space of Lipschitz functions (a dense subset of the space of continuous functions $C(X)$ on $X$). Since points of $X$ are extreme points of $\mathcal S(C(X))$ and the restriction of the metric $\rho_{L_d}$ to $X$ coincide with $d$, the data $(C(X), L_d)$ is equivalent to data $(X, d)$. Thus, the notion of a compact metric space $(X, d)$ motivates a compact quantum metric space $(C(X), L_d)$.

Generalizing these ideas from $C(X)$ to a non-commutative $C^*$-algebra, an important class of compact quantum metric spaces has been considered in literature (see \cite{connes, matrix}) by virtue of a Lip-norm on a dense subset of the order-unit space of self adjoint elements of a $C^*$-algebra. Note that the state space of the order-unit space of self adjoint elements of a $C^*$-algebra and the state space of the $C^*$-algebra coincide. So, we can start with a seminorm $L$ on a $C^*$-algebra $\mathcal A$, taking finite values on a dense subset of $\mathcal A$, which satisfies the properties of Definition \ref{1} and $L(a^*) = L(a)$ for all $a\in\mathcal A$. A simple argument (as mentioned in section 2 of \cite{2004}) shows that $L$, and the restriction of $L$ to the order-unit space of self adjoint elements of $\mathcal A$, determine the same metric on $\mathcal S(\mathcal A)$. In this paper, we shall be using $(\mathcal A, L)$ as a notation for the corresponding compact quantum metric space. Although an order-unit space is a vector space over $\mathbb R$ and a $C^*$-algebra is a vector space over the complex field $\mathbb C$, nevertheless the justification for the notation $(\mathcal A, L)$ for the compact quantum metric space comes from the fact above. Most of the classical examples of the subjects arise from $C^*$-algebras. For example, see \cite{connes} for the compact quantum metric spaces arising from a spectral triple and \cite{2002} for compact quantum metric spaces arising from ergodic strongly continuous action of a compact group on $\mathcal A$ by automorphism. For more details about compact quantum metric spaces, see \cite{survey, 1999}.

Taking motivation from the notion of Gromov-Hausdorff distance between two compact metric spaces \cite{Gromov}, Rieffel also introduced the notion of the \emph{quantum Gromov-Hausdorff distance} between two compact quantum metric spaces in \cite{2004}. Let $(\mathcal A, L_{\mathcal A})$ and $(\mathcal B, L_{\mathcal B})$ be two compact quantum metric spaces. Let $\mathcal M(L_{\mathcal A}, L_{\mathcal B})$ denote the set of Lip-norms on $A\oplus B$ that induce $L_{\mathcal A}$ and $L_{\mathcal B}$ on $\mathcal A$ and $\mathcal B$ respectively.

\begin{definition}[\cite{2004}, Definition 4.2] The quantum Gromov-Hausdorff distance $\dist_q(\mathcal A, \mathcal B)$ between $(\mathcal A, L_{\mathcal A})$ and $(\mathcal B, L_{\mathcal B})$,  is defined as
$$\dist_q(\mathcal A,\mathcal B) = \inf\{\dist_{\rho_L} (\mathcal S(\mathcal A), \mathcal S(\mathcal B)) : L\in\mathcal M(L_{\mathcal A}, L_{\mathcal B})\},$$ where $\dist_{\rho_L} (\mathcal S(\mathcal A)$ denotes the classical Gromov-Hausdorff distance between (compact subsets) $\mathcal S(\mathcal A)$ and $\mathcal S(\mathcal B)$ in the metric space $(\mathcal S(A\oplus B), \rho_L)$.
\end{definition}

This is not just an extension of concepts of classical compact metric spaces, but these gave mathematical justification for the assertions found in theoretical physics literature which deal with string theory and related parts of quantum field theory, that the complex matrix algebras converge to two-sphere $S^2$ (or to related spaces). This has been explored in detail by Rieffel in \cite{matrix}. More examples of convergence in the quantum Gromov-Hausdorff distance can be found in \cite{podles, BD, 2005}. To understand this convergence, it is important to understand elements of $\mathcal M(L_{\mathcal A}, L_{\mathcal B})$. In \cite{2004}, Rieffel introduced the notion of \emph{bridges} that we recall below.

\begin{definition}[\cite{2004}, Definition 5.1] A bridge between $(\mathcal A, L_{\mathcal A})$ and $(\mathcal B, L_{\mathcal B})$ is a
seminorm, N on $\mathcal A\oplus\mathcal B$ such that
\begin{enumerate}
\item $N$ is continuous for the norm on $\mathcal A\oplus\mathcal B$,
\item $N(e_{\mathcal A}, e_{\mathcal B}) = 0$ but $N(e_{\mathcal A}, 0)\neq 0$.
\item For any $a\in\mathcal A$ and $\delta>0$, there is an element $b\in\mathcal B$ such that
$$\max\{L_{\mathcal B}(b), N(a, b)\}\leq L_{\mathcal A}(a)+\delta,$$
and similarly for $\mathcal A$ and $\mathcal B$ interchanged.
\end{enumerate}
\end{definition}

It was proved in Theorem 5.2 of \cite{2004} that if $N$ is a bridge between $(\mathcal A, L_{\mathcal A})$ and $(\mathcal B, L_{\mathcal B})$ and $L$ is defined as $$L(a, b) = \max\{L_{\mathcal A}(a), L_{\mathcal B}(b), N(a,b)\},$$ then $L\in\mathcal M(L_{\mathcal A}, L_{\mathcal B})$. In \cite{matrix}, Rieffel found a sequence of complex matrix algebras (arising from  finite-dimensional representations of $SU(2)$) converging to the space of continuous complex valued functions on the two-sphere $S^2$. In this paper, we have found a compact quantum metric space structure on the Toeplitz algebras on {\em generalized Bergman spaces} on the closed unit ball in $\mathbb C^d$ such that a sequence of these algebras converges to the space of continuous functions on the sphere in $\mathbb R^{2d}$, denoted by $C(S^{2d-1})$, in the quantum Gromov-Hasudorff distance.

Motivated by the quantum Gromov-Hasudorff distance, various notions of convergence have been introduced in the literature of noncommutative geometry. Order-unit quantum Gromov–Hausdorff distance was introduced by Li in \cite{Li}. The notion of Gromov-Hausdorff propinquity was introduced by Latr\'{e}moli\`ere, see \cite{AF, 2015, propinquity}. For the notion of the quantum Gromov-Hausdorff propinquity, see \cite{2016}. The notion of Gromov-Hausdorff propinquity for metric spectral triples can be found in \cite{2022}. The notion of Gromov-Hausdorff convergence of state spaces for spectral truncations is of interest to mathematicians, see \cite{preprint, Suijlekom}. For a general setup of spectral truncations in noncommutative geometry and operator systems, see \cite{connes-van}. For operator systems with Lip-norm, even a matricial version of the quantum Gromov-Hasudorff distance can be defined, see \cite{Kerr}. The convergence of the particular sequence obtained in this paper can also be described in terms of the matricial quantum Gromov-Hasudorff distance which we shall remark on in the last section.

In Section \ref{section2}, we explain the space of Toeplitz algebras on the generalized Bergman space and describe the compact quantum metric space structure on these spaces. In Section \ref{section3}, we prove our main theorem that the sequence of Toeplitz algebras converges to odd spheres. In Section \ref{remarks}, a few remarks are mentioned.

%%%%%%%%%%%%%%%%%%%%%%%%%
\section{Toeplitz algebras on generalized Bergman spaces}\label{section2}

Let $d\in\mathbb N$ be fixed. Let $B^{2d}$ denotes the open unit ball in $\mathbb C^d$. For all $n\geq d$, let $dV_n$ denote the volume measure on $B^{2d}$ given by $$dV_n = c_n(1-|z|^2)^{n-d}\, dV,$$where $dV$ denotes the Lebesuge measure and $c_n=\dfrac{n!}{(n-d)!\ \pi^d}$ is a normalizing constant so that $dV_n$ is a probability measure on $B^{2d}$. We consider
$$\mathcal H_n = \left\{ f : B^{2d}\rightarrow \mathbb C : f\text{ is analytic and } \int_{B^{2d}} |f|^2\, dV_n<\infty\right\}.$$
Then $\mathcal H_n$ with the inner product $\bra f|g\ket_{\mathcal H_n} = \int_{B^{2d}} f\overline{g} \, dV_n$ is the reproducing kernel Hilbert space with the kernel function $$K_n(z,w) = \dfrac{1}{(1-\bra z|w\ket)^{n+1}} \text{ for all } z,w\in\mathbb C^d.$$ For $n=d$, the space $\mathcal H_n$ is known as the Bergman space and for all $n> d$, we call it a generalized Bergman space. For more details, see Chapter 2 of \cite{Bergman}. For each $n\geq d$, an orthonormal basis for $\mathcal H_n$ is given by $(e_{k,n})$, where $$e_{k,n} = \left(\dfrac{(|k|+n)!}{k!\ n!}\right)^{1/2} z^k,$$ where $k=(k_1, \dots, k_d)$ is an $d$-tuple of non-negative integers and we take $|k| = k_1+\dots+k_d$\ , $k! = k_1!\dots k_d!$\ , $z^k = z_1^{k_1}\dots z_d^{k_d}$.

Let $\bar{B}^{2d}$ denote the closed unit ball in $\mathbb C^d$. For $\phi\in C(\bar{B}^{2d})$, we define the Toeplitz operator $T_{\phi,n} : \mathcal H_n\rightarrow\mathcal H_n$ as $$T_{\phi, n}(f)=P_n(\phi f),$$ where $P_n: L^2(B^{2d}, dV_n) \rightarrow \mathcal H_n$ is the orthogonal projection.

Let $\mathcal T_n$ be the $C^*$-subalgebra of operators on $\mathcal H_n$ generated by $\{T_{\phi, n} : \phi\in C(\bar{B}^{2d})\}$. In Theorem 1 of \cite{odd}, it was proved that the Toeplitz algebra $\mathcal T_d$ contains the space of compact operators $\mathscr K(\mathcal H_d)$ and
$$\mathcal T_d =\{T_{\phi, d}+K : \phi\in C(\bar{B}^{2d}) \text{ and } K\in\mathscr K(\mathcal H_d)\}.$$
 The quotient algebra $\mathcal T_d/\mathscr K(\mathcal H_d)$ is $C^*$-isomorphic to $C(S^{2d-1})$ with the isomorphism given by $$T_{\phi,d}+ K \xrightarrow[]{\pi} \phi|_{S^{2d-1}} \text{ for all } \phi\in C(\bar{B}^{2d}), K\in\mathscr K(\mathcal H_d).$$
Using techniques used in \cite{odd}, we prove that the same holds for $\mathcal T_n$.
\begin{theorem}\label{2} For all $n\geq d$, we have $$\mathcal T_n =\{T_{\phi, n}+K : \phi\in C(\bar{B}^{2d}) \text{ and } K\in\mathscr K(\mathcal H_n)\},$$ and the quotient $C^*$-algebra $\mathcal T_n/\mathscr K(\mathcal H_n)$ is $C^*$-isomorphic to $C(S^{2d-1})$ with the isomorphism given by $$T_{\phi, n}+ K \xrightarrow[]{\pi} \phi|_{S^{2d-1}} \text{ for all } f\in C(\bar{B}^{2d}), K\in\mathscr K(\mathcal H_n).$$ So, we have the existence of the following short exact sequence : \begin{equation}0\rightarrow \mathscr K(\mathcal H_n) \xrightarrow[]{i} \mathcal T_n\xrightarrow[]{\pi} C(S^{2d-1})\rightarrow 0,\label{11}\end{equation} where $i$ denotes the inclusion map.
\end{theorem}
\begin{proof}
For $n=d$, the theorem was proved in Theorem 1 of \cite{odd}. Let $n>d$ be fixed and let $U: \mathcal H_n\rightarrow \mathcal H_d$ be the unitary transformation given by $U(e_{k,n})=e_{k,d}.$ Along the lines of the proof of Lemma 3 of \cite{odd}, we have that $U^*T_{\phi, d} U - T_{\phi, n}$ is a compact operator. It follows that the commutator ideal $\mathcal C_n$ of $\mathcal T_n$ is contained in $\mathscr K(\mathcal H_n)$. Along the lines of the proof of Lemma 1 of \cite{odd}, $\mathcal T_n$ is an irreducible $C^*$-algebra. Using Theorem 1.4.2 of \cite{Arveson},  $\mathcal C_n=\mathscr K(\mathcal H_n)$. Since $\{T_{\phi, n}+K : \phi\in C(\bar{B}^{2d}) \text{ and } K\in\mathscr K(\mathcal H_n)\}$ is a $C^*$-algebra, we get $$\mathcal T_n=\{T_{\phi, n}+K : \phi\in C(\bar{B}^{2d}) \text{ and } K\in\mathscr K(\mathcal H_n)\}.$$
The rest of the proof is along the lines of the proof of Theorem 1 of \cite{odd}.
\end{proof}

 By virtue of the short exact sequence \eqref{11} in Theorem \ref{2}, we define a compact quantum metric space structure on $\mathcal T_n$ for all $n\geq d$ using Theorem 3.4 of \cite{toeplitz}. Let $n\geq d$ be fixed. Let $(e_j)_{j\geq 1}$ be an enumeration of the orthonormal basis of $\mathcal H_n$ mentioned above. We equip $\mathscr K(\mathcal H_n)$ with the compact quantum metric space structure $(Lip(\mathscr K(\mathcal H_n))\oplus \mathbb RI, \tilde{L}_n)$ given by
\begin{align*} Lip(\mathscr K(\mathcal H_n)) =\{T\in \mathscr K(\mathcal H_n) & : T^*=T, \langle Te_i|e_j\rangle_{\mathcal H_n}\in\mathbb R \text{ for all } i,j\geq 1 \\
&\text{ and }  \sup_{i, j\geq 1} (i+j)^{n+2} |\langle Te_i|e_j\rangle_{\mathcal H_n}|<\infty\},\end{align*}
with $\tilde{L}_n(T) = \sup_{i, j\geq 1} (i+j)^{n+2} \big|\langle Te_i|e_j\rangle_{\mathcal H_n}\big|$ for all $T\in Lip(\mathscr K(\mathcal H_n))$.

By Theorem 3.4 of \cite{toeplitz}, $(Lip(\mathscr K(\mathcal H_n))\oplus \mathbb RI, \tilde{L}_n)$ is a compact quantum metric space. Now, we consider a fixed positive linear splitting $\sigma : C(S^{2d-1})\rightarrow \mathcal T_n$ of the short exact sequence \eqref{11}, for example, $f\in C(S^{2d-1})$ is mapped to $T_{\tilde{f}}$ where $\tilde{f}$ is the unique solution of the Dirichlet's problem. Then using Theorem 3.6 of \cite{toeplitz}, we have the folllowing compact quantum compact space structure on $\mathcal T_n$.
\begin{theorem}\label{44} The space of Toeplitz operators on the generalized Bergman space $\mathcal T_n$ has the compact quantum space structure $(Lip(\mathcal T_n), L_n)$ where
$$Lip(\mathcal T_n)=Lip (\mathscr K(\mathcal H_n))\oplus \{T_{\sigma(f)} : f \text{ real valued Lipschitz function on } S^{2d-1}\},$$ and $L_n(T_f + K) = \tilde{L}_n(K) + L(f|_{S^{2d-1}})$ for all $T_f + K \in Lip(\mathcal T_n)$ and $L(f|_{S^{2d-1}})$ is the Lipschitz norm of $f|_{S^{2d-1}}$.\end{theorem}

Now we claim that $\mathcal T_n$ converges to $C(S^{2d-1})$ with the compact quantum structure of real valued Lipschitz functions with the Lipschitz norm in the quantum Gromov-Hausdorff distance.

%%%%%%%%%%%%%%%%%%%%%%%%%
\section{Main theorem}\label{section3}

Now, we prove our main theorem. The idea of the proof is the same as that used by Rieffel in \cite{matrix}.
\begin{theorem}\label{3} The sequence of compact quantum metric spaces $(Lip(\mathcal T_n), L_n)$ converges to $C(S^{2d-1})$ in the quantum Gromov-Hausdorff distance.
\end{theorem}

The following lemma is useful.

\begin{lemma}\label{22} For all $n\geq d$ and for all $T\in\mathcal T_n$, we have $$\|T-\sigma(\pi(T))\|\leq \gamma_n L_n(T),$$ where $\gamma_n$ is a decreasing sequence converging to $0.$
\end{lemma}
\begin{proof} Let $T=T_{\sigma(f)}+K\in \mathcal T_n$ for some $K\in Lip(\mathscr K(\mathcal H_n))$ and a real valued Lipschitz function $f$ on $C(S^{2d-1})$. Then we have,
\begin{align*}\|T-\sigma(\pi(T))\| = \|K\| & \leq \sup_{j\geq 1} \sum_{i\geq 1} \big|\langle Te_i|e_j\rangle_{\mathcal H_n}\big| \\
& \leq \tilde{L}_n(K)\sup_{j\geq 1} \sum_{i\geq 1}(i+j)^{-n-2}\\
& \leq \tilde{L}_n(K) \sum_{i\geq 1}(i+1)^{-n-2}\\
& \leq L_n(T) (\zeta(n+2)-1),
\end{align*} where $\zeta$ denotes the Riemann-Zeta function. The proof is completed by taking $\gamma_n =\zeta(n+2)-1$.
\end{proof}

\textit{Proof of Theorem \ref{3}}. Let $\eps>0$. Let $n_0\in\bN$ be such that $\gamma_n\leq \eps/2$ for $n\geq n_0$.

For all $T\in \mathcal T_n$ and $f\in C(S^{2d-1})$, we define  $N(T, f) = \gamma_{n_0}^{-1} \|\pi(T)-f\|$  and $$\tilde{L}(T, f) = \max\{L_n(T), L(f), N(T, f)\}.$$ Since $L_n(\sigma(f))=L(f)$ and $N(\sigma(f), f)=0$, we have $$\tilde{L}(\sigma(f), f) = L(f).$$
Also, $L(\pi(T)) \leq L_n(T)$ and $N(T, \pi(T)) = 0$. Hence,
$$\tilde{L}(T, \pi(T)) = L_n(T).$$
So, $N$ is bridge between the compact quantum metric spaces $\mathcal T_n$ and $C(S^{2d-1})$. By Theorem 5.2 of \cite{2004}, $\tilde{L}$ is a Lip-norm on $\mathcal T_n\oplus C(S^{2d-1})$ that induces the previously defined Lip-norms on the compact quantum metric spaces $\mathcal T_n$ and $C(S^{2d-1})$.

By Proposition 1.3 of \cite{matrix}, we know that the state space of $C(S^{2d-1})$, denoted by $\mathcal S_{C(S^{2d-1})}$, is contained in the $\gamma_{n_0}$ neighbourhood of $\mathcal S_{\mathcal T_n}$ for $\rho_{\tilde{L}}$. So,  $\mathcal S_{C(S^{2d-1})}$ is in the $\eps/2$ neighbourhood of $\mathcal S_{\mathcal T_n}$ for $\rho_{\tilde{L}}$.

Let $T\in \mathcal T_n$ and $f\in C(S^{2d-1})$ such that $\tilde{L}(T, f)\leq 1$ and $\nu\in\mathcal S_{\mathcal T_n}$. We consider $\mu = \nu\circ\sigma\in\mathcal S_{C(S^{2d-1})}$. For all $n\geq n_0$, we have
\begin{align*} |\mu(T, f)-\nu(T,f)| & = |\nu(T-\sigma(f))|\\
& \leq \|T-\sigma(\pi(T))\| + \|\sigma(\pi(T))-\sigma(f))\| \\
&\leq \|T-\sigma(\pi(T))\| + \|\pi(T))-f\| \\
& \leq 2\gamma_{n_0}.
\end{align*}
The last inequality uses Lemma \ref{22} and the fact that $\tilde{L}(T, f)\leq 1$ (this implies that $\|\pi(T))-f\|\leq \gamma_{n}$ and $L_n(T)\leq 1$). Hence for all $n\geq n_0$, we have that $\mathcal S_{\mathcal T_n}$ is contained in the $\eps$ neighbourhood of $\mathcal S_{C(S^{2d-1})}$ for $\rho_{\tilde{L}}$.

Hence, we get that the quantum Gromov-Hausdorff distance between the compact quantum metric spaces on $\mathcal T_n$ and $C(S^{2d-1})$ is less than or equal to $\eps$ for all $n\geq n_0$, that is,  $\mathcal T_n$ converges to $C(S^{2d-1})$ in the quantum Gromov-Hausdorff distance.\qed

%%%%%%%%%%%%%%%%%%%%%%%%%
\section{Remarks}\label{remarks}

\textbf{Remark 4.1}. Since $\mathcal T_n$ and $C(S^{2d-1})$ are unital $C^*$-algebras, these are also examples of Lip-normed unital $C^*$-algebras with the compact quantum metric space structures given in Theorem \ref{44}. Using Theorem 3.9 and Theorem 3.11 of \cite{Paulsen}, the maps $\pi$ and $\sigma$ are completely positive maps. Using arguments similar to the ones in Example 3.12 of \cite{Kerr}, Lemma \ref{22} holds at each matrix level and we get that $\mathcal T_n$ converges to $C(S^{2d-1})$ in the complete distance also, as defined in \cite{Kerr}.

\textbf{Remark 4.2}. With the same methods as in the proof of Theorem \ref{3}, we can get a more general result. Let $\mathcal H_n$ be a sequence of separable Hilbert spaces. Let $\mathcal A$ be a $C^*$-algebra with a compact quantum metric space structure. If $\mathcal B_n$ be a sequence of $C^*$-algebras such that there exists a split exact sequence of $C^*$-algebra isomorphisms
$$0\rightarrow \mathscr K(\mathcal H_n) \rightarrow \mathcal B_n\rightarrow \mathcal A,$$
then there exists a compact quantum metric space structure on $\mathcal B_n$ such that $\mathcal B_n$ converges to $\mathcal A$ in quantum Gromov-Hausdorff distance.

\textbf{Remark 4.3}. For each $\alpha\in[d,\infty)$, we can consider the volume measure $dV_\alpha$ on $B^{2d}$ given by $$dV_\alpha =\dfrac{\Gamma(\alpha+1)}{\Gamma(\alpha-d+1)\pi^d}(1-|z|)^{n-\alpha}\, dV.$$ Then the space $\mathcal H_\alpha$ defined as the set of all analytic functions on $B^{2d}$ which are square-integrable with respect to $dV_\alpha$ and the Toeplitz algebra $\mathcal T_\alpha$ on $\mathcal H_\alpha$ are well defined. Since arguments in this paper and \cite{toeplitz} work when natural numbers are replaced by a positive real number, we get the following extension of Theorem \ref{3} : For $\alpha\in[d,\infty)$, the net of compact quantum metric spaces $(Lip(\mathcal T_\alpha), L_\alpha)$ converges to $C(S^{2d-1})$ in the quantum Gromov-Hausdorff distance.

%%%%%%%%%%%%%%%%%%%%%%%%%
\textbf{Acknowledgments}

We would like to thank Sutanu Roy for many useful discussions. The research of the authors is supported by JCB/2021/000041 of SERB, India.

\end{document}